\documentclass[12pt,a4paper]{article}
\usepackage[dvips]{graphicx}
\usepackage{epsf}
\usepackage{amssymb}
\usepackage{pstricks}
\usepackage{pst-node}
\usepackage{lineno}
\usepackage[normalem]{ulem}

\usepackage{mathrsfs}
\usepackage{amsmath}
\usepackage{amsfonts}

\usepackage{tabularx}
\usepackage{booktabs}
\usepackage{multirow}

\newenvironment{block}[3]{%
  \vspace{2mm}\par%
  \refstepcounter{#2}%
  \noindent\textbf{%
    #1 \thesection.\arabic{#2}.
  }%
}{%
  \par%
}

\newenvironment{thm}[1][]{%
  \begin{block}{\textsc{Theorem}}{theorem}{#1}%
  \slshape%
}{%
  \end{block}%
  \vspace{2mm}%
}
\newenvironment{lem}[1][]{%
  \begin{block}{\textsc{Lemma}}{theorem}{#1}%
  \slshape%
}{%
  \end{block}%
  \vspace{2mm}%
}

\newenvironment{wn}[1][]{%
  \begin{block}{\textsc{Corollary}}{theorem}{#1}%
  \slshape%
}{%
  \end{block}%
  \vspace{2mm}%
}

\newenvironment{obs}[1][]{%
  \begin{block}{\textsc{Observation}}{theorem}{#1}%
  \slshape%
}{%
  \end{block}%
  \vspace{2mm}%
}

\newenvironment{proof}[1][Proof]{%
  \par\noindent\emph{#1}.
}{%
  \eop
  \bigskip%
}

\newenvironment{proof2}[1][Proof]{%
  \par\noindent\emph{#1}.
}{%
  \bigskip%
}

\renewcommand\caption[1]{\small\refstepcounter{figure}%
\begin{center}\textbf{Fig.\ \thefigure .}\ #1\end{center}\normalsize}

\newcommand{\eop}{\hspace*{\fill}\nolinebreak$\Box$\nolinebreak\par}
\newcommand{\later}[1]{{}}
\newcommand{\old}[1]{{}}

\newrgbcolor{gray5}{0.5 0.5 0.5}
\newrgbcolor{gray7}{0.7 0.7 0.7}
\newrgbcolor{gray9}{0.9 0.9 0.9}

\def\cer{{\rm cer}}
\def\c{{\rm c}}
\def\supp{s}
\def\leaf{l}

\title{{\textsc{\Certified domination}}}
\date{\today}
\begin{document}

\begin{center}
\noindent\textbf{\Large\uppercase{Certified domination}}
\end{center}

\vspace{10mm}
\noindent\textsc{Magda Dettlaff, Magdalena Lema\'{n}ska}\\[1mm]
Gda\'nsk University of Technology, 80-233 Gda\'nsk, Poland\\
{\small \texttt{\{mdettlaff,magda\}@mif.pg.gda.pl}}\\[3mm]
\textsc{Jerzy Topp, Rados\l{}aw Ziemann, Pawe\l{}~\.Zyli\'nski}\\[1mm]
University of Gda\'{n}sk, 80-952 Gda\'nsk, Poland\\
{\small \texttt{\{j.topp,rziemann,zylinski\}@inf.ug.edu.pl}}

\begin{abstract}
\noindent Imagine that we are given a set $D$ of officials and a set $W$ of civils.
For each civil $x \in W$, there must be an official $v \in D$
that can serve~$x$, and whenever any such $v$ is serving~$x$,
there must also be another civil $w \in W$ that
observes~$v$, that is, $w$ may act as a kind of witness,
to avoid any abuse from $v$. What is the minimum number of officials
to guarantee such a service, assuming a given social network?

In this paper, we introduce the concept of certified domination
that perfectly models the aforementioned problem.
Specifically, a~dominating set $D$ of a graph $G=(V_G,E_G)$ is said to be certified
if every vertex in $D$ has either zero or at least two neighbours in
$V_G\setminus D$. The cardinality of a~minimum certified dominating set in
$G$ is called the certified domination number of $G$.
Herein, we present the exact values of the
certified domination number for some classes of graphs
as well as provide some upper bounds on this parameter for arbitrary graphs.
We then characterise a wide class of graphs with equal domination
and certified domination numbers and characterise graphs with large values
of certified domination numbers. Next,
we examine the effects on the certified domination number
when the graph is modified by deleting/adding an edge or a vertex.
We also provide Nordhaus-Gaddum
type inequalities for the certified domination number.
Finally, we show that the (decision) certified domination problem
is NP-complete.
As a side result,
we characterise a wider class of $D\!D_2$-graphs,
thus generalizing a result of~\cite{HR13}.

{\bf \AmS \; Subject Classification:} 05C05, 05C
\end{abstract}

\section{Introduction}
Imagine that we are given a set $D$ of officials and a set $W$ of civils.
For each civil $x \in W$, there must be an official $v \in D$
that can serve~$x$, and whenever any such $v$ is serving~$x$,
there must also be another civil $w \in W$ that
observes~$v$, that is, $w$ may act as a kind of witness,
to avoid any abuse from $v$. What is the minimum number of officials
to guarantee such a service, assuming a given social network?

The aforementioned problem motivates us
introducing the concept of certified domination.
Specifically, let $D$ be a subsets of the vertex set of a graph $G=(V_G,E_G)$.
We say that $D$ {\em dominates} $G$ (or is a {\em dominating set} of $G$) if each vertex in the set $V_G \setminus D$ has a~neighbour in $D$.
The cardinality of a~minimum dominating set in $G$ is called the~{\em
domination number of $G$} and denoted by $\gamma(G)$,
and any minimum dominating set of $G$ is called a {\em $\gamma$-set}.
A dominating set $D$ of $G$ is called {\em certified}
if every vertex $v \in D$ has either zero or at least two neighbours in
$V_G\setminus D$.
The cardinality of a~minimum certified dominating set in $G$ is called the {\em
certified domination number of $G$} and denoted
by $\gamma_\cer(G)$.
A~minimum certified dominating set of $G$ is called a
{\em $\gamma_\cer$-set}.
Notice that, by the definition, $V_G$
is a~certified dominating set of $G$, and certainly 
$1\leq \gamma_{\cer}(G)\leq |V_G|$. Furthermore, one can observe that
$\gamma_{\cer}(G)\neq |V_G|-1$.

There is a wealth of literature about
domination and its variations in graphs;
we refer to the excellent books of Haynes,
Hedetniemi, and Slater~\cite{HHS98, HHS98b}. 
The domination concept we introduce perfectly fits into that area where, for a given graph $G$,
domination parameters are defined by imposing additional constraints
on a dominating set $D$ or its complement $V_G \setminus D$.
This area includes, to mention but a few,
the multiple domination, the distance domination, or  the global domination. 
In particular,
the problem of certified domination is closely related
to the problem of
existence a $DD_2$-pair in a graph, introduced by Henning and Rall in~\cite{HR13}.
Recall, a set $X \subseteq V_G$ of vertices is {\em $2$-dominating} in~$G$
if it is a dominating set of $G$ and every vertex in $V_G\setminus X$
has at least two 
neighbours in $X$~\cite{FJ85a,FJ85b}.
A {\em $D\!D_2$-pair} of $G$ is a pair $(D, D_2)$ of disjoint sets of vertices of $G$ such that
$D$ is a dominating set of $G$ and $D_2$ is a $2$-dominating set of $G$;
a graph that has a $D\!D_2$-pair is called a {\em $D\!D_2$-graph}.
One can observe that if $G$ has a $D\!D_2$-pair $(D,D_2)$, then
the set $D$ is a certified dominating set. 
However, there are graphs $G$ with
$\gamma_\cer(G) < |D|$ for any $(D,D_2)$-pair in $G$ (if any),
see~Fig.~\ref{fig:gammacer_vs_DD2} for an illustration.

\begin{figure}[h!]
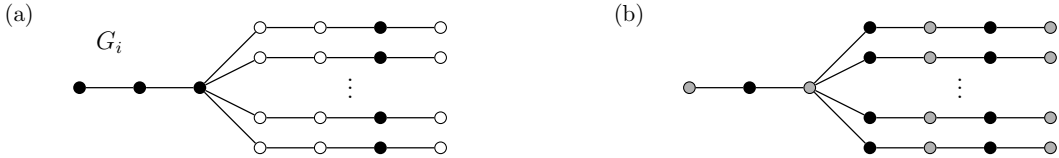

\begin{center}

\pspicture(0,-0.2)(8,2.4)\scalebox{0.8}{

\cnode*[linewidth=0.5pt,fillstyle=solid,fillcolor=white,linecolor=black](0,1){3pt}{v1}
\cnode*[linewidth=0.5pt,fillstyle=solid,fillcolor=white,linecolor=black](1,1){3pt}{v2}
\cnode*[linewidth=0.5pt,fillstyle=solid,fillcolor=white,linecolor=black](2,1){3pt}{v3}

\cnode[linewidth=0.5pt,fillstyle=solid,fillcolor=white,linecolor=black](3,2){3pt}{x1}
\cnode[linewidth=0.5pt,fillstyle=solid,fillcolor=white,linecolor=black](4,2){3pt}{x2}
\cnode*[linewidth=0.5pt,fillstyle=solid,fillcolor=white,linecolor=black](5,2){3pt}{x3}
\cnode[linewidth=0.5pt,fillstyle=solid,fillcolor=white,linecolor=black](6,2){3pt}{x4}

\cnode[linewidth=0.5pt,fillstyle=solid,fillcolor=white,linecolor=black](3,1.5){3pt}{y1}
\cnode[linewidth=0.5pt,fillstyle=solid,fillcolor=white,linecolor=black](4,1.5){3pt}{y2}
\cnode*[linewidth=0.5pt,fillstyle=solid,fillcolor=white,linecolor=black](5,1.5){3pt}{y3}
\cnode[linewidth=0.5pt,fillstyle=solid,fillcolor=white,linecolor=black](6,1.5){3pt}{y4}

\cnode[linewidth=0.5pt,fillstyle=solid,fillcolor=white,linecolor=black](3,0.5){3pt}{z1}
\cnode[linewidth=0.5pt,fillstyle=solid,fillcolor=white,linecolor=black](4,0.5){3pt}{z2}
\cnode*[linewidth=0.5pt,fillstyle=solid,fillcolor=white,linecolor=black](5,0.5){3pt}{z3}
\cnode[linewidth=0.5pt,fillstyle=solid,fillcolor=white,linecolor=black](6,0.5){3pt}{z4}

\cnode[linewidth=0.5pt,fillstyle=solid,fillcolor=white,linecolor=black](3,0){3pt}{w1}
\cnode[linewidth=0.5pt,fillstyle=solid,fillcolor=white,linecolor=black](4,0){3pt}{w2}
\cnode*[linewidth=0.5pt,fillstyle=solid,fillcolor=white,linecolor=black](5,0){3pt}{w3}
\cnode[linewidth=0.5pt,fillstyle=solid,fillcolor=white,linecolor=black](6,0){3pt}{w4}

\ncline[linewidth=0.6pt]{v2}{v1}
\ncline[linewidth=0.6pt]{v2}{v3}
\ncline[linewidth=0.6pt]{x1}{v3}
\ncline[linewidth=0.6pt]{y1}{v3}
\ncline[linewidth=0.6pt]{z1}{v3}
\ncline[linewidth=0.6pt]{w1}{v3}

\ncline[linewidth=0.6pt]{x1}{x2}
\ncline[linewidth=0.6pt]{x3}{x2}
\ncline[linewidth=0.6pt]{x3}{x4}

\ncline[linewidth=0.6pt]{z1}{z2}
\ncline[linewidth=0.6pt]{z3}{z2}
\ncline[linewidth=0.6pt]{z3}{z4}

\ncline[linewidth=0.6pt]{y1}{y2}
\ncline[linewidth=0.6pt]{y3}{y2}
\ncline[linewidth=0.6pt]{y3}{y4}

\ncline[linewidth=0.6pt]{w1}{w2}
\ncline[linewidth=0.6pt]{w3}{w2}
\ncline[linewidth=0.6pt]{w3}{w4}

\rput(.5,1.75){$G_i$}

\rput(4.5,1.1){$\vdots$}

\rput(-1,2.2){\small (a)}
}
\endpspicture
\pspicture(0,-0.2)(4.8,1)\scalebox{0.8}{

\cnode[linewidth=0.5pt,fillstyle=solid,fillcolor=gray7,linecolor=black](0,1){3pt}{v1}
\cnode*[linewidth=0.5pt,fillstyle=solid,fillcolor=white,linecolor=black](1,1){3pt}{v2}
\cnode[linewidth=0.5pt,fillstyle=solid,fillcolor=gray7,linecolor=black](2,1){3pt}{v3}

\cnode*[linewidth=0.5pt,fillstyle=solid,fillcolor=white,linecolor=black](3,2){3pt}{x1}
\cnode[linewidth=0.5pt,fillstyle=solid,fillcolor=gray7,linecolor=black](4,2){3pt}{x2}
\cnode*[linewidth=0.5pt,fillstyle=solid,fillcolor=white,linecolor=black](5,2){3pt}{x3}
\cnode[linewidth=0.5pt,fillstyle=solid,fillcolor=gray7,linecolor=black](6,2){3pt}{x4}

\cnode*[linewidth=0.5pt,fillstyle=solid,fillcolor=white,linecolor=black](3,1.5){3pt}{y1}
\cnode[linewidth=0.5pt,fillstyle=solid,fillcolor=gray7,linecolor=black](4,1.5){3pt}{y2}
\cnode*[linewidth=0.5pt,fillstyle=solid,fillcolor=white,linecolor=black](5,1.5){3pt}{y3}
\cnode[linewidth=0.5pt,fillstyle=solid,fillcolor=gray7,linecolor=black](6,1.5){3pt}{y4}

\cnode*[linewidth=0.5pt,fillstyle=solid,fillcolor=white,linecolor=black](3,0.5){3pt}{z1}
\cnode[linewidth=0.5pt,fillstyle=solid,fillcolor=gray7,linecolor=black](4,0.5){3pt}{z2}
\cnode*[linewidth=0.5pt,fillstyle=solid,fillcolor=white,linecolor=black](5,0.5){3pt}{z3}
\cnode[linewidth=0.5pt,fillstyle=solid,fillcolor=gray7,linecolor=black](6,0.5){3pt}{z4}

\cnode*[linewidth=0.5pt,fillstyle=solid,fillcolor=white,linecolor=black](3,0){3pt}{w1}
\cnode[linewidth=0.5pt,fillstyle=solid,fillcolor=gray7,linecolor=black](4,0){3pt}{w2}
\cnode*[linewidth=0.5pt,fillstyle=solid,fillcolor=white,linecolor=black](5,0){3pt}{w3}
\cnode[linewidth=0.5pt,fillstyle=solid,fillcolor=gray7,linecolor=black](6,0){3pt}{w4}

\ncline[linewidth=0.6pt]{v2}{v1}
\ncline[linewidth=0.6pt]{v2}{v3}
\ncline[linewidth=0.6pt]{x1}{v3}
\ncline[linewidth=0.6pt]{y1}{v3}
\ncline[linewidth=0.6pt]{z1}{v3}
\ncline[linewidth=0.6pt]{w1}{v3}

\ncline[linewidth=0.6pt]{x1}{x2}
\ncline[linewidth=0.6pt]{x3}{x2}
\ncline[linewidth=0.6pt]{x3}{x4}

\ncline[linewidth=0.6pt]{z1}{z2}
\ncline[linewidth=0.6pt]{z3}{z2}
\ncline[linewidth=0.6pt]{z3}{z4}

\ncline[linewidth=0.6pt]{y1}{y2}
\ncline[linewidth=0.6pt]{y3}{y2}
\ncline[linewidth=0.6pt]{y3}{y4}

\ncline[linewidth=0.6pt]{w1}{w2}
\ncline[linewidth=0.6pt]{w3}{w2}
\ncline[linewidth=0.6pt]{w3}{w4}

\rput(4.5,1.1){$\vdots$}

\rput(-1,2.2){\small (b)}
}
\endpspicture
\caption{The family of graphs $G_i$. (a) Black vertices form a certified dominating set $D_\c$ with $|D_\c|=i+3$, $i \ge 2$.
(b) Black and grey vertices form a~$(D,D_2)$-pair, respectively, with $|D|=2i+1$. Observe that
if $i \ge 3$, then $G_i$ has no $(D,D_2)$-pair with $|D| \le i+3$.}\label{fig:gammacer_vs_DD2}
\end{center}
\end{figure}

\vspace{-10mm}
\paragraph{Organization of the paper.}
In Section~\ref{sec:Pre}, we present the exact values of the
certified domination number for some elementary classes of graphs.
Some upper bounds on this new parameter for an arbitrary graph are presented in Section~\ref{sec:upper}.
Then, in Section~\ref{sec:gammacer_gamma} and Section~\ref{sec:gammacer_large}, respectively,
we characterise a wide class of graphs with equal domination
and certified domination numbers and characterise graphs with large values
of certified domination numbers.
Next, in Section~\ref{sec:influence},
we examine the effects on the certified domination number
when the graph is modified by deleting/adding an edge or a vertex.
Finally, Section~\ref{sec:NG} is devoted to Nordhaus-Gaddum
type inequalities for the certified domination number, while
in Section~\ref{sec:NP-hardness},
we show that the (decision) certified domination problem
is \mbox{NP-complete}.
In addition, as a side result, in Section~\ref{sec:DD2},
we characterise a wider class of $D\!D_2$-graphs,
thus generalizing a result of~\cite{HR13}.

\subsection{Definitions and notation}
For general graph theory terminology, we follow~\cite{Die12}.
In particular,
for a vertex $v$ of a~graph $G=(V_G,E_G)$,
its {\em $($open$)$ neighbourhood\/}, denoted by $N_{G}(v)$, is the set of all vertices adjacent to $v$,
and the cardinality of $N_G(v)$, denoted by $\deg_G(v)$, is called the {\em degree} of~$v$.
The {\em closed neighbourhood\/} of $v$, denoted by $N_{G}[v]$, is the set $N_{G}(v)\cup \{v\}$.
In general, for a subset $X\subseteq V_G$ of vertices,
 the {\em $($open$)$ neighbourhood\/} of $X$, denoted by $N_{G}(X)$,
is defined to be $\bigcup_{v\in X}N_{G}(v)$, and the {\em closed\/} neighbourhood of $X$,
denoted by $N_{G}[X]$, is the set $N_{G}(X)\cup X$.
The minimum and maximum degree of a vertex in $G$ is denoted by $\delta(G)$ and $\Delta(G)$, respectively.
A vertex of degree $|V_G|-1$ is called a {\em universal} vertex of $G$.
 A vertex of degree one is called a {\em leaf}, and the only neighbour of a leaf is called its {\em support vertex} (or simply, its {\em support}).
 If a support vertex
 has at least two leaves as neighbours,  we call it a {\em strong\/} support, 
 otherwise it is a {\em weak} support. 
The set of leaves of $G$ is denoted by $L_G$.
For a leaf $v \in L_G$, its support vertex is denoted by $\supp_G(v)$,
and for a weak support $v$, the unique leaf adjacent to $v$
is denoted by $\leaf_G(v)$.
The set of weak supports of $G$ is denoted by $S_1(G)$, while
the set of strong supports of $G$ is denoted by $S_2(G)$.

\section{Elementary graph classes}\label{sec:Pre}
We begin by presenting the exact values of the
certified domination number for some elementary classes of graphs.
\begin{obs}\label{obs:Pn}
If $P_n$ is an $n$-vertex path, then
\[\gamma_{\cer}(P_n)= \left\{\begin{array}{cl}
1 & \mbox{if $n=1$ or $n=3$;}\\
2 & \mbox{if $n=2$;}\\
4 & \mbox{if $n=4$;}\\
\lceil \frac{n}{3}\rceil & \mbox{otherwise.} \end{array}\right.\]
\end{obs}
\begin{obs}\label{obs:Cn}
If $C_n$ is an $n$-vertex cycle, $n \ge 3$, then
$\gamma_{\cer}(C_n)=\lceil \frac{n}{3}\rceil$.
\end{obs}
\begin{obs}\label{obs:Kn}
If $K_n$ is an $n$-vertex complete graph, then
\[\gamma_{\cer}(K_n)= \left\{\begin{array}{cl}
1 & \mbox{if $n=1$ or $n \ge 3$;}\\
2 & \mbox{if $n=2$.} \end{array}\right.\]
\end{obs}
\begin{obs}\label{obs:Knm}
If $K_{m,n}$ is a complete bipartite graph with $1\le m \le n$, then
\[\gamma_{\cer}(K_{m,n})= \left\{\begin{array}{cl}
1 & \mbox{if $m=1$ and $n>1$;}\\
2 & \mbox{otherwise.} \end{array}\right.\]
\end{obs}
\begin{obs}\label{obs:Wn}
If $W_n$ is an $n$-vertex wheel, then $\gamma_\cer(W_n)=1$.
\end{obs}

In addition, we have the following two general observations
on the certified domination number of a graph.

\begin{obs}\label{obs:universal}
If $G$ is a graph of order at least three, then
$\gamma_{\cer}(G)=1$ if and only if $G$ has a universal vertex.
\end{obs}
\begin{obs}\label{obs:many_compo}
If $G_1,\ldots,G_k$ are the connected components of a graph $G$,
then $\gamma_\cer(G)=\sum\limits_{i=1}^k\gamma_\cer(G_i)$.
\end{obs}

\section{Upper bounds on the certified domination number}\label{sec:upper}
In this section we focus on 
upper bounds on the certified domination number. 
We start with two simple observations and then
present our main result of this section: an upper bound on $\gamma_\cer(G)$ with respect to the domination number $\gamma(G)$
and the number $|S_1(G)|$ of weak supports in $G$.

\begin{obs}\label{obs:support}
If $D_\c$ is a certified dominating set of a graph $G$, then
every support vertex of $G$ belongs to $D_\c$.
\end{obs}

\begin{proof}
Let $s$ be a support vertex of $G$. 
If  $s$ were not in $D$, then $\leaf_G(s)$ should be in $D_\c$.
But then $\leaf_G(s)$ would have only one neighbour
in $V_G \setminus D_\c$,
and $D_\c$ would not be a certified dominating set.
\end{proof}

\begin{obs}\label{obs:strong_support}
Let $G$ be a graph of order $n$.
If the strong supports of $G$
are adjacent to $k$ leaves in total,
then $\gamma_{\cer}(G) \le n-k$.
In particular, $\gamma_{\cer}(G) \le n-2 \, |S_2(G)|$.
\end{obs}

\begin{proof}
Let $L$ be the set of all leaf-neighbours of strong supports of $G$.
Then $|L|=k$ and the set $V_G \setminus L$ is a certified dominating set of $G$.
Thus $\gamma_{\cer}(G) \le |V_G \setminus L| = n-k \le n-2 \, |S_2(G)|$ as  $|L| \ge 2\, |S_2(G)|$.
\end{proof}

Before we present our main result,
let us introduce some useful terminology.
Let $D$ be a dominating set of a graph $G$.
An element of $D$ that has all neighbours in $D$ is said to be {\em shadowed} with respect to $D$ (shortly {\em shadowed}\/), an
element of $D$ that has exactly one neighbour in $V_G \setminus D$
is said to be {\em half-shadowed} with respect to $D$ (shortly {\em half-shadowed}\/), while an element of $D$ having at least two neighbours in $V_G \setminus D$
is said to be {\em illuminated} with respect to $D$ (shortly {\em illuminated}\/).
It is easy to observe that if $D$ is a minimum dominating set of a graph with no isolated vertices, then $D$ has no shadowed element, 
and if $D$ is a certified dominating set, then
$D$ has no half-shadowed element. 

\begin{thm}\label{thm:gammacer_vs_gamma_con}
If $G$ is a connected graph, then $\gamma_{\rm cer}(G) \le \gamma(G) +|S_1(G)|$.
\end{thm}
\begin{proof}
If $G$ is a graph of order at most two, then the inequality is obvious.
Thus assume that $G$ has at least three vertices.
Let $D$ be a $\gamma$-set of $G$ that minimizes the number of
half-shadowed vertices 
and such that $D$ does not contain any leaf of $G$. (Notice that such $D$ always exists as $G$ is connected and $|V_G|\ge 3$.)
Let $D_{\rm hs} \subseteq D$ be the set of all half-shadowed
vertices of $D$.
If $D_{\rm hs}=\emptyset$, then $\gamma_\cer(G)=\gamma(G)
\le \gamma(G)+|S_1(G)|$. Thus assume that $D_{\rm hs} \not = \emptyset$.
\begin{description}
\item[Claim 1.] \textsl{If $v \in D_{\rm hs}$, then $\deg_G(v) \ge 2$ and $v \notin S_2(G)$.}
\\[2mm] {The inequality $\deg_G(v) \ge 2$ follows from the choice of $D$, that is, from the assumption that
$D \cap L_G = \emptyset$. 
To argue the second property, suppose on the contrary that $v$ is a strong support.
Again, since $L_G \cap D = \emptyset$ and $v$ has at least two leaf-neighbours,
$v$ would not be half-shadowed, a~contradiction.}
\end{description}
Next we show that all half-shadowed vertices are weak supports.
Suppose on the contrary that there is a half-shadowed vertex $v \in D_{\rm hs} \setminus S_1(G)$
and let $u$ be the unique neighbour of $v$ in $V_G \setminus D$.
Since $v$ is neither a weak nor strong support (by assumption and Claim 1, respectively), it implies that
$u$ is not a leaf.
Furthermore, we have the following claim.
\begin{description}
\item[Claim 2.] \textsl{All but $v$ neighbours of $u$ are in $V_G \setminus D$.}
\\[2mm] {Otherwise the set $D \setminus \{v\}$
would be a smaller (than $D$) dominating set of $G$.}
\end{description}
On the other hand, as regards all (but $u$) neighbours of $v$ that are in $D$,
we have the following claim.
\begin{description}
\item[Claim 3.] \textsl{If $w \in N_G(v) \setminus \{u\}$, then $w$ is not shadowed.}
    \\[2mm] {Otherwise the set $D \setminus \{w\}$
would be a smaller (than $D$) dominating set of $G$.}
\end{description}
Consequently, keeping in mind the fact that none of neighbours of $v$
is a leaf (see Claim~1), by combining Claims~2 and~3,
we conclude that the set $(D \setminus\{v\}) \cup \{u\}$ would
be a dominating set with a smaller number of half-shadowed vertices, a contradiction.
Therefore, the set $D_{\rm hs}$ of half-shadowed vertices 
consists of weak supports of $G$ only.

Observe now that adding to $D$ all leaves adjacent
to half-shadowed weak supports results in a dominating
set $D'$ of $G$ with no half-shadowed vertices, 
that is, $D'$ is a~certified dominating set of $G$.
Therefore $\gamma_{\rm cer}(G) \le  |D'| = |D| + |D_{\rm hs}| = \gamma(G) + |D_{\rm hs}| \le \gamma(G) +|S_1(G)|$.
\end{proof}

From Observation~2.\ref{obs:many_compo}
and Theorem~3.\ref{thm:gammacer_vs_gamma_con},
we immediately obtain the following corollary.

\begin{wn}\label{wn:gammacer_vs_gamma}
If $G$ is a graph, then $\gamma_{\rm cer}(G) \le \gamma(G) +|S_1(G)|$.
\end{wn}

Finally, we have  the following corollary.

\begin{wn}\label{wn:gammacer_vs_2gamma}
If $G$ is a graph, then $\gamma_{\rm cer}(G) \le 2 \, \gamma(G)$.
\end{wn}
\begin{proof}
Let $H$ be a connected component of $G$.
If $H=K_2$, then $\gamma_\cer(H)=2\gamma(H)$.
If $H \neq K_2$, then $|S_1(H)| \le \gamma(H)$,
and thus $ \gamma_{\rm cer}(H) \le 2 \, \gamma(H)$ by
Theorem~3.\ref{thm:gammacer_vs_gamma_con}.
Consequently, taking into account Observation~2.\ref{obs:many_compo},
we conclude that  $\gamma_{\rm cer}(G) \le 2 \, \gamma(G)$.
\end{proof}

We emphasize that the above upper bounds are sharp:
the bound $\gamma_\cer(G) \le \gamma(G)+|S_1(G)|$ in terms of $|S_1(G)|$,
while the bound $\gamma_\cer(G) \le 2 \gamma(G)$ in terms of $\gamma(G)$,
and their sharpnesses are established by coronas of graphs
which we shall discuss in Section~\ref{sec:gammacer_large} (see Remark on page~\pageref{remark}).

\section{Graphs with $\gamma_\cer=\gamma$}\label{sec:gammacer_gamma}

We continue our study on the certified domination number
by focusing now on the class of graphs with $\gamma_\cer=\gamma$.
When trying to characterise this class,
one may expect that the main problem lies in leaves of a graph.
In fact, from the inequalities
$\gamma(G)\le \gamma_\cer(G) \le \gamma(G)+|S_1(G)|$
(see Corollary~3.\ref{wn:gammacer_vs_gamma}),
we immediately have the first two results of this section.

\begin{wn}\label{wn:gammacer_is_gamma1}
If $G$ is a graph with no weak support, then $\gamma_\cer(G)=\gamma(G)$.
\end{wn}
\begin{wn}\label{wn:gammacer_is_gamma_delta}
If $G$ is a graph with $\delta(G)\ge 2$, then $\gamma_\cer(G)=\gamma(G)$.
\end{wn}
The above two corollaries also follow from the next more general lemma.
\begin{lem}\label{thm:S1_gammacer_gamma}
If a connected graph $G$ has at least three vertices, then $\gamma_\cer(G)=\gamma(G)$ if and only if there exists a minimum dominating set $D$ of $G$ such that $N_G(s)\setminus L_G\not\subseteq D$ for every $s \in S_1(G)$. \end{lem}

\begin{proof}  Assume first that $\gamma_\cer(G)=\gamma(G)$. Let $D_\c$ be a minimum
certified dominating set of~$G$. Then $D_\c$ is a minimum dominating set of $G$. Now, if
$s\in S_1(G)$, 
then $D_\c\cap\{s, \leaf_G(s)\}\not= \emptyset$ (as $D_\c$ is dominating in $G$), $|D_\c\cap\{s, \leaf_G(s)\}|\not=2$ (otherwise $D_\c\setminus\{\leaf_G(s)\}$ would be a smaller dominating set of $G$), and $D_\c\cap\{s, \leaf_G(s)\}\not= \{\leaf_G(s)\}$ (otherwise $\leaf_G(s)$
would be half-shadowed). Thus $D_\c\cap\{s, \leaf_G(s)\}= \{s\}$ and $(N_G(s) \setminus L_G)\cap (V_G\setminus D_\c)= (N_G(s)\setminus \{\leaf_G(s)\})\cap (V_G\setminus D_\c)\not=\emptyset$ (otherwise $s$
would be half-shadowed), and so $N_G(s)\setminus L_G\not\subseteq D_\c$.

Assume now that in $G$ there exists a $\gamma$-set $D$ such that $N_G(s)\setminus L_G \not\subseteq D$ for every $s\in S_1(G)$. Of all such sets, choose one, say $D'$, that does not contain any leaf of $G$ (such $D'$ exists in every connected graph of order at least three) and minimizes the number of its half-shadowed vertices.
We claim that such $D'$ is a certified dominating set of $G$ (and therefore
$\gamma(G)= |D'|= \gamma_{\cer}(G)$). Suppose, on the contrary, that some element $v$ of $D'$ is half-shadowed. Let $v'$ be the unique element of $N_G(v)\setminus D'$.
Since $v$ is half-shadowed, $v\not\in S_2(G)$, and $v\not\in S_1(G)$ (as every element of $S_1(G)$ is illuminated by the adjacent leaf and, by the assumption, by at least one non-leaf). Finally, since $D'\cap L_G=\emptyset$ (by the choice of $D'$) and $v\in D'$, we have $v\not\in
L_G$ and $d_G(v)\ge 2$. Now, if it were $N_G(v')\cap (D'\setminus\{v\})\not= \emptyset$, then $D'\setminus\{v\}$ would be a dominating set of $G$ smaller than $D'$, a contradiction. Thus $N_G(v') \setminus\{v\}$ must be a~nonempty subset of $V_G\setminus D'$ and, then, $D''=(D'\setminus \{v\}) \cup \{v'\}$ is a minimum dominating set of $G$ and it has less half-shadowed vertices than $D'$, a final contradiction which proves that $\gamma(G)=\gamma_{\cer}(G)$. \end{proof}

Observe that if $G=\overline{K_n}$, then $\gamma_\cer(G)  = n = \gamma(G)$. Next, if $G=lK_2$,
then $\gamma_\cer(G) =2l \neq l = \gamma(G)$.
In the latter case, $S_1(G)=V_G=L_G$ and
$G$ has no minimum dominating set $D$ of $G$
such that $N_G(s)\setminus L_G\not\subseteq D$ for every $s \in S_1(G)$. Therefore, taking into account Observation~2.\ref{obs:many_compo}
and Lemma~4.\ref{thm:S1_gammacer_gamma},
we obtain the following corollary for graphs which are not necessarily connected.

 \begin{wn}\label{thm:S1_gammacer_gamma_dis}
If $G$ is a graph, then $\gamma_\cer(G)=\gamma(G)$ if and only if there exists a minimum dominating set $D$ in $G$ such that $N_G(s)\setminus L_G\not\subseteq D$ for every $s \in S_1(G)$.
\end{wn}
Furthermore, we have the following relation between graphs each of which has a unique minimum dominating set and those for which $\gamma_\cer$ and $\gamma$ are equal.
\begin{wn}\label{wn:gammacer_is_gamma2}
If a graph $G$ has a unique minimum dominating set, then $\gamma_\cer(G)=\gamma(G)$.
\end{wn}

\begin{proof}
If $S_1(G)=\emptyset$, then $\gamma_\cer(G) = \gamma(G)$ by Corollary~4.\ref{wn:gammacer_is_gamma1}. Thus assume that
$S_1(G) \neq \emptyset$. Let $D$ be the minimum dominating set of $G$. From the uniqueness and minimality of $D$ it follows that $S_1(G) \subseteq D$ and $L_G \subseteq V_G \setminus D$.
Now, if it were $\gamma_\cer(G) \neq \gamma(G)$, then,
by Lemma~4.\ref{thm:S1_gammacer_gamma},
we could find $s \in S_1(G)$
such that $N_G(s) \setminus \{\leaf_G(s)\} \subseteq D$, and
then the set $(D \setminus \{s\}) \cup \{\leaf_G(s)\}$ would be
another  minimum dominating set of~$G$, which is impossible.
\end{proof}

\section{Graphs with large values of $\gamma_\cer$}\label{sec:gammacer_large}

As we have already observed, for any graph $G$ of order $n$, $\gamma_\cer(G) \le n$, $\gamma_\cer(G) \neq n-1$,
and there are graphs $G$ with $\gamma_\cer(G)=n$,
for example,
the complement of a complete graph $K_n$
or a 4-vertex path $P_4$.
Thus it is natural to try to characterise all graphs with $\gamma_\cer=n$
and $\gamma_\cer=n-2$, respectively,
which is successfully carried out in this section.
In particular, we prove that $\gamma_\cer(G)=n$ if and only
if $G$ is the complement of a complete graph,
the corona of a graph, or the union of both of them.
Recall, the {\em corona product} (or simply, the {\em corona}) of two graphs $H$ and $F$
is the graph $G = H \circ F$ resulting from
the disjoint union of $H$ and $|V_H|$
copies of $F$ in which the $i$-th vertex of $H$ is joined
to all vertices of the $i$-th copy of $F$. 
If $F$ is a 1-vertex graph, $F = K_1$, then
the corona $H \circ K_1$ is simply called the {\em corona} of $H$.

\begin{lem}\label{lem:coronas-n} Let $G$ be a connected graph of order $n$. If $G$ is the corona of some graph, then $\gamma_\cer(G)=n$.
\end{lem}
\begin{proof} Let $D_\c$ be a smallest certified dominating set of $G$.
It suffices to prove that $D_\c=V_G$. This is obvious if $n=2$. Thus assume
$n>2$. In this case, since $G$ is the corona of some graph, every vertex
of $G$ either is a leaf of $G$ or is adjacent to exactly one leaf of $G$.
From this and from  Observation~3.\ref{obs:support} it follows that
$V_G \setminus L_G\subseteq D_\c$. Moreover, every leaf $l$ of $G$ also belongs to $D_\c$ (as otherwise its only neighbour $\supp_G(l)$ would be half-shadowed). Consequently, $L_G\subseteq V_G$ and therefore $D_\c=V_G$.
\end{proof}

\begin{lem}\label{lem:arbitrary}
Let $G$ be a connected graph of order $n \ge 2$.
If $\gamma_\cer(G)=n$, then $G$ is the corona of some graph.
\end{lem}
\begin{proof}
The statement is obvious for connected graphs of order at most $4$.
Thus assume that $G$ is a connected graph of order $n \ge 5$
and $\gamma_\cer(G)=n$.
Now, since $\gamma(G) \le n/2$ for every graph with
no isolated vertex,
so by Corollary~3.\ref{wn:gammacer_vs_2gamma}
we have
$\gamma_\cer(G) \le 2 \gamma(G) \le n=\gamma_\cer(G)$.
Thus $\gamma(G)= n/2$ and so  $G$ is the corona of some graph
 (as it was proved in~\cite{FJKR85}).
\end{proof}

From the above lemmas, we immediately conclude with the following theorem.

\begin{thm}\label{thm:corona}
If $G$ is a graph of order $n$, then
$\gamma_\cer(G)=n$ if and only if $G$
is either the complement of a complete graph,
or the corona of a graph, 
or the union of both of them.
\end{thm}

\noindent\textbf{Remark.}\label{remark} We incidentally observe that
the above result implies the sharpness of the upper bound
$\gamma_\cer(G) \le \gamma(G)+|S_1(G)|$ (Corollary~3.\ref{wn:gammacer_vs_gamma})
in terms of $|S_1(G)|$ as well as the sharpness
of the upper bound $\gamma_\cer(G) \le 2 \gamma(G)$ (Corollary~3.\ref{wn:gammacer_vs_2gamma}) in terms of
$\gamma(G)$, since for the corona $G$ of any graph,
we have $|S_1(G)|=\gamma(G)$ and $\gamma_\cer(G)=2 \gamma(G)$.

\subsection{Graphs with $\gamma_\cer=n-2$}\label{subsec:n2}
A {\em diadem graph of a graph $H$} is a graph obtained from the corona $H \circ K_1$
by adding a~new vertex, say $v$, and joining $v$ to one of support vertices of $H\circ K_1$ (see Fig.~\ref{fig:example_diadem_chandelier}).
%
\begin{figure}[h!]
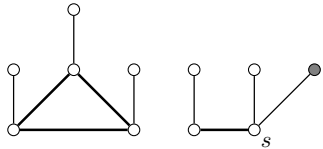

\begin{center}

\pspicture(0,-0.2)(4,2)\scalebox{0.8}{

\cnode[linewidth=0.5pt,fillstyle=solid,fillcolor=white,linecolor=black](0,0){3pt}{x1}
\cnode[linewidth=0.5pt,fillstyle=solid,fillcolor=white,linecolor=black](2,0){3pt}{x2}
\cnode[linewidth=0.5pt,fillstyle=solid,fillcolor=white,linecolor=black](4,0){3pt}{x3}
\cnode[linewidth=0.5pt,fillstyle=solid,fillcolor=white,linecolor=black](0,1){3pt}{x4}
\cnode[linewidth=0.5pt,fillstyle=solid,fillcolor=white,linecolor=black](1,1){3pt}{x5}
\cnode[linewidth=0.5pt,fillstyle=solid,fillcolor=white,linecolor=black](2,1){3pt}{x6}
\cnode[linewidth=0.5pt,fillstyle=solid,fillcolor=white,linecolor=black](3,0){3pt}{x7}
\cnode[linewidth=0.5pt,fillstyle=solid,fillcolor=white,linecolor=black](4,1){3pt}{x8}
\cnode[linewidth=0.5pt,fillstyle=solid,fillcolor=gray,linecolor=black](5,1){3pt}{x9}
\cnode[linewidth=0.5pt,fillstyle=solid,fillcolor=white,linecolor=black](1,2){3pt}{x10}
\cnode[linewidth=0.5pt,fillstyle=solid,fillcolor=white,linecolor=black](3,1){3pt}{x11}

\ncline[linewidth=0.6pt]{x1}{x4}
\ncline[linewidth=1.2pt]{x1}{x5}
\ncline[linewidth=1.2pt]{x1}{x2}
\ncline[linewidth=1.2pt]{x2}{x5}
\ncline[linewidth=1.2pt]{x3}{x7}
\ncline[linewidth=0.6pt]{x3}{x8}
\ncline[linewidth=0.6pt]{x3}{x9}
\ncline[linewidth=0.6pt]{x5}{x10}
\ncline[linewidth=0.6pt]{x6}{x2}
\ncline[linewidth=0.6pt]{x7}{x11}

\rput(4.2,-0.2){\small $s$}


}
\endpspicture
\caption{The diadem graph resulting from the corona $G=(K_3 \cup K_2) \circ K_1$ by adding a leaf to the support vertex $s$ of $G$.}\label{fig:example_diadem_chandelier}
\end{center}
\end{figure}



\begin{lem}\label{lem:diadem_n2}
If $G$ is a diadem graph of order $n$, then $\gamma_{\cer}(G)=n-2$. \end{lem}

\begin{proof}Let $s$ be the unique strong support of $G$, and let $l_1,l_2$
be the two leaves of $G$ adjacent to $s$ in $G$. It is obvious that $V_G
\setminus \{l_1,l_2\}$ is a certified dominating set of $G$. Let $D_\c$ be a smallest certified dominating set of $G$. Then $V_G \setminus L_G\subseteq D_\c$
(by Observation~3.\ref{obs:support}) and $\{l_1,l_2\}\cap D_\c=\emptyset$. Moreover,
every leaf $l$ different from $l_1$ and $l_2$  belongs to $D_\c$ (otherwise $s_G(l)$
would be half-shadowed). Consequently $D_\c=V_G\setminus \{l_1,l_2\}$ and therefore $\gamma_\cer(G)=n-2$.\end{proof}

\begin{lem}\label{abc} Let $G$ be a connected graph of order $n$. If $\gamma_{\cer}(G)=n-2$, then  $G=C_3$, $G=C_4$, or $G$ is a diadem graph
$($of a connected graph$)$. \end{lem}

\begin{proof2}
If $G$ is a connected graph of order at most $n \le 4$ and $\gamma_\cer(G)=n-2$, then $G=K_{1,2}$, $G=C_3$ or $G=C_4$. Thus assume that $n \ge 5$. In this case $\delta (G)=1$, as otherwise, since $\gamma_{\cer}(G)=\gamma(G)$ (by  Corollary~4.\ref{wn:gammacer_is_gamma_delta}), $\gamma_{\cer}(G)=n-2$,  and $\gamma(G)\leq \frac{n}{2}$, we would have  $n-2=\gamma_{\cer}(G)=\gamma(G)\leq \frac{n}{2}$, which is impossible. We now claim that $G$ is
a~diadem graph.

By way of contradiction, suppose that the claim is false. Let $G$ be a smallest counterexample, say of order $n$ ($n\ge 5$), such that  $\gamma_{\cer}(G)=n-2$ and
$G$ is not a diadem graph.
 Let $D_\c$ be a $\gamma_\cer$-set of $G$, and let $v$ and $u$ be the only elements of $V_G \setminus D_\c$. From the fact that $D_\c=V_G\setminus \{v,u\}$ is a certified dominating set
of $G$ it follows that if $x\in D_\c$, then either $x\in N_G(v)\cap N_G(u)$ or $x\not\in
N_G(v)\cup N_G(u)$. This proves that $N_G(v)\cap D_\c= N_G(u)\cap D_\c$. In addition, the set
$V_G\setminus N_G[\{v,u\}]$ is nonempty, as otherwise $\{v,u\}$ would be a certified dominating set of $G$ and we would have $n-2=\gamma_\cer(G)\le |\{v,u\}|=2$, which is impossible.

Let $G'$ denote the subgraph $G- N_G[\{v,u\}]$ of $G$. From the assumption $\gamma_\cer(G) = n-2$ it easily follows that $\gamma_\cer(G')=|V_{G'}|$.
Thus, by
Theorem~5.\ref{thm:corona}, every connected component of $G'$ is an isolated vertex
or the corona of a graph.


Let $H$ be a connected component of $G'$. From the fact that
$D_\c=V_G\setminus \{v,u\}$ is a~minimum certified dominating set of $G$ it follows that
at least one vertex of $H$ is not adjacent to any vertex belonging to $N_G[\{v,u\}]
\setminus \{v,u\}$ as otherwise $D_\c\setminus V_H$ would be a~certified dominating set of $G$, which is impossible as $\gamma_\cer(G)\le |D_\c\setminus V_H|<|D_\c|=\gamma_\cer(G)$. From this we conclude that $G'$ has no isolated vertex. Consequently, every connected component of $G'$ is the corona of a graph.

We now claim that $K_2$ is not a connected component of $G'$. Suppose on the contrary that $K_2$ on vertices  $a$ and $b$ is a connected component of $G'$. Then one of the vertices $a$ and $b$ is a leaf in $G$ and the latter one is adjacent to a vertex in $N_G[\{v,u\}]\setminus\{v,u\}$, say $a\in L_G$ and $b$ is adjacent to a vertex $w\in N_G[\{v,u\}]\setminus\{v,u\}$. Let $\widetilde{G}$ denote the graph $G-\{a, b\}$ (of order $n-2$). For this graph either $\gamma_\cer(\widetilde{G})< n-4$, or $\gamma_\cer(\widetilde{G})=n-4$, or $\gamma_\cer(\widetilde{G})>n-4$. Assume first that $\gamma_\cer(\widetilde{G})< n-4$.
Let $\widetilde{D}_\c$ be a smallest certified dominating set of $\widetilde{G}$. Then
$\widetilde{D}_\c\cup \{b\}$ (if $(N_G(b)\setminus\{a\})\setminus  \widetilde{D}_\c \not=\emptyset$) or $\widetilde{D}_\c\cup \{a,b\}$ (if $N_G(b)\setminus\{a\}\subseteq \widetilde{D}_\c$) is a certified dominating set of $G$ and $\gamma_\cer(G)\le |\widetilde{D}_\c\cup \{a,b\}|= \gamma_\cer(\widetilde{G})+2< n-2$, a contradiction.
Assume now that $\gamma_\cer(\widetilde{G})>n-4$. Then $\gamma_\cer(\widetilde{G})=n-2
=|V_{\widetilde{G}}|$ and, by Theorem~5.\ref{thm:corona},  $\widetilde{G}$ is
the corona of a graph. But this is impossible as no vertex of $N_G[\{v,u\}] \setminus \{v,u\}$ is a leaf or a neighbour of exactly one leaf. Finally, assume that $\gamma_\cer( \widetilde{G})=n-4 = |V_{\widetilde{G}}|-2$. In this case the choice of $G$ implies
that $\widetilde{G}$ is the diadem graph in which $v$ and $u$ are leaves and $w$ is
their only common neighbour. Now, it is obvious that the graph $G$ (obtained from $\widetilde{G}$ by the addition of the vertices $a$ and $b$, and the edges $ab$
and $bw$) is a diadem graph, a~contradiction.

Now, to complete the proof, it suffices to show that this smallest counterexample
is not a counterexample, that is, it suffices to show that $G$ is a diadem graph.
It is enough to prove that: (1) no vertex belonging to $N_G[\{v,u\}]$ is adjacent to
a leaf of a connected component of $G'$ of order at least four, (2) $v$ and $u$ have exactly one common neighbour, and (3) $v$ and $u$ are not adjacent in $G$.

\begin{itemize}
\item[(1)] Suppose on the contrary that there is a vertex in $N_G[\{v,u\}]$ adjacent to a leaf $l$ of a connected component $H$ (of order at least four) of $G'$. Let $L$ be the set of leaves of $H$ within the distance at most $2$ from $s_H(l)$. Then $D=D_\c \setminus \big(L \cup \{s_{H}(l)\}\big)$ is a~certified dominating set of $G$ and $|D|<|D_\c|$, a contradiction.

\item[(2)]   Suppose on the contrary that $|N_G[\{v,u\}] \setminus \{v,u\}| \ge 2$.
 Let us consider the set $S=\{x\in V_{G'}
\colon N_G(x)\cap N_G(\{v,u\})\not= \emptyset\}$. By (1), $S$ is a subset of
$V_{G'} \setminus L_{G'}$. In addition, since $G'$ is the corona of a graph,
every vertex of $S$ is adjacent to a vertex of $L_{G'}$. From the supposition
$|N_G[\{v,u\}] \setminus \{v,u\}| \ge 2$ and from properties of elements of $S$
it follows that $D= \{v,u\}\cup (V_{G'} \setminus L_{G'})\cup (L_{G'}\setminus
N_{G'}(S))$ ($=\{v,u\}\cup (V_{G'} \setminus (N_G(S)\cap L_G))$) is a~certified dominating set
of $G$ and $|D|<|D_\c|$, a  contradiction.

\item[(3)]  Suppose on the contrary that $vu \in E_G$, and consider the graph $G''=G-vu$ of
order $n$, in which, by (2), $v$ and $u$ are leaves, and they have exactly one common neighbour, say $w$. In this graph we have either $\gamma_\cer(G'') >n-2$ (and therefore
$\gamma_\cer(G'') =n$), or $\gamma_\cer(G'')=n-2$, or $\gamma_\cer(G'')<n-2$.
Assume first that $\gamma_\cer(G'')=n$. Then, by Theorem~5.\ref{thm:corona}, $G''$ is the corona of a graph, but this is impossible as leaves $v$ and $v$ share the same neighbour $w$. Assume now that $\gamma_\cer(G'')=n-2$. Then, by the choice of $G$,  $G''$ is a diadem graph. Let $L$ be the set of leaves of $G''$ within the distance at most $3$ from $v$ (and $u$). Then $D=(D_\c\setminus (L \cup \{w\})) \cup \{v\}$ is a certified dominating set of $G$ and $|D|<|D_\c|$, a contradiction.
Finally, assume that $\gamma_\cer(G'')<n-2$. Let $D''_\c$ be a smallest
certified dominating of $G''$.
Since $w$ is a strong support of $G''$,
$w \in D''_\c$ by Observation~3.\ref{obs:support},
and $v,u \notin D''_\c$ by minimality of $D''_\c$.
But then, $D''_\c$ is also a~certified dominating set of $G$ and so $\gamma_\cer(G) < n-2$,
a~final contradiction. \eop \end{itemize}\end{proof2}

From Theorem~5.\ref{thm:corona}, Lemma~5.\ref{lem:diadem_n2} and Lemma~5.\ref{abc}, we have the final characterisation of graphs of order $n$ with $\gamma_\cer=n-2$.

\begin{thm}\label{thm:diadem}
Let $G$ be a graph of order $n \ge 3$. Then
$\gamma_\cer(G)=n-2$ if and only if $G$ is $C_3, C_4$, or a diadem graph,
or $G$ is one of these three graphs with possible number of isolated vertices,
or $G$ is the union of one of these three graphs
with the corona of some graph, with possible number of isolated vertices.\eop
\end{thm}

\section{Influence of deleting/adding edge/vertex}\label{sec:influence}
In this section, following~\cite{BHHH15,BVV15,CFFS12,EG12},
to mention but a recent few,
we examine the effects on the certified domination number
when the graph is modified by deleting/adding an edge or a vertex.
We observe that deleting an edge or a vertex may arbitrarily increase the certified domination number.
For example, for the graph $G_i$ of order $2i+4$ illustrated in Fig.~\ref{fig:edge_removing_adding}(a)
we have $\gamma_\cer(G_i)=i+1$ and $\gamma_\cer(G_i - e)= 2i+4$.
To argue a similar influence of deleting a vertex, consider a wheel graph $W_n$ with the hub $v$. We have $\gamma_\cer(W_n)=1$
and $\gamma_\cer(W_n - v)=\lceil \frac{n-1}{3}\rceil$.

\begin{figure}[!h]
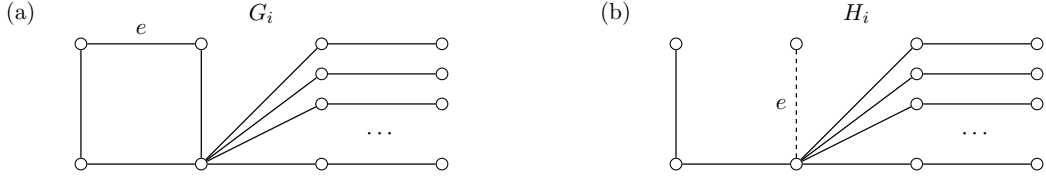

\begin{center}
\pspicture(0,-0.2)(7.8,2.2)\scalebox{0.8}{
\cnode[linewidth=0.5pt,fillstyle=solid,fillcolor=white,linecolor=black](0,2){3pt}{v1}
\cnode[linewidth=0.5pt,fillstyle=solid,fillcolor=white,linecolor=black](0,0){3pt}{v2}
\cnode[linewidth=0.5pt,fillstyle=solid,fillcolor=white,linecolor=black](2,0){3pt}{v3}
\cnode[linewidth=0.5pt,fillstyle=solid,fillcolor=white,linecolor=black](2,2){3pt}{v4}
\cnode[linewidth=0.5pt,fillstyle=solid,fillcolor=white,linecolor=black](4,0){3pt}{l11}
\cnode[linewidth=0.5pt,fillstyle=solid,fillcolor=white,linecolor=black](6,0){3pt}{l12}
\cnode[linewidth=0.5pt,fillstyle=solid,fillcolor=white,linecolor=black](4,1){3pt}{l21}
\cnode[linewidth=0.5pt,fillstyle=solid,fillcolor=white,linecolor=black](6,1){3pt}{l22}
\cnode[linewidth=0.5pt,fillstyle=solid,fillcolor=white,linecolor=black](4,1.5){3pt}{l31}
\cnode[linewidth=0.5pt,fillstyle=solid,fillcolor=white,linecolor=black](6,1.5){3pt}{l32}
\cnode[linewidth=0.5pt,fillstyle=solid,fillcolor=white,linecolor=black](4,2){3pt}{l41}
\cnode[linewidth=0.5pt,fillstyle=solid,fillcolor=white,linecolor=black](6,2){3pt}{l42}

\ncline[linewidth=0.6pt]{v2}{v1}
\ncline[linewidth=0.6pt]{v2}{v3}
\ncline[linewidth=0.6pt]{v4}{v3}
\ncline[linewidth=0.6pt]{v4}{v1}
\ncline[linewidth=0.6pt]{v3}{l11}
\ncline[linewidth=0.6pt]{l12}{l11}
\ncline[linewidth=0.6pt]{v3}{l21}
\ncline[linewidth=0.6pt]{l22}{l21}
\ncline[linewidth=0.6pt]{v3}{l31}
\ncline[linewidth=0.6pt]{l32}{l31}
\ncline[linewidth=0.6pt]{v3}{l41}
\ncline[linewidth=0.6pt]{l42}{l41}

\rput(1,2.25){\small $e$}
\rput(5,.5){\small $\cdots$}

\rput(3,2.5){\small $G_i$}
\rput(-1,2.5){\small (a)}

}
\endpspicture
\pspicture(0,-0.2)(4.8,1.84)\scalebox{0.8}{

\cnode[linewidth=0.5pt,fillstyle=solid,fillcolor=white,linecolor=black](0,2){3pt}{v1}
\cnode[linewidth=0.5pt,fillstyle=solid,fillcolor=white,linecolor=black](0,0){3pt}{v2}
\cnode[linewidth=0.5pt,fillstyle=solid,fillcolor=white,linecolor=black](2,0){3pt}{v3}
\cnode[linewidth=0.5pt,fillstyle=solid,fillcolor=white,linecolor=black](2,2){3pt}{v4}
\cnode[linewidth=0.5pt,fillstyle=solid,fillcolor=white,linecolor=black](4,0){3pt}{l11}
\cnode[linewidth=0.5pt,fillstyle=solid,fillcolor=white,linecolor=black](6,0){3pt}{l12}
\cnode[linewidth=0.5pt,fillstyle=solid,fillcolor=white,linecolor=black](4,1){3pt}{l21}
\cnode[linewidth=0.5pt,fillstyle=solid,fillcolor=white,linecolor=black](6,1){3pt}{l22}
\cnode[linewidth=0.5pt,fillstyle=solid,fillcolor=white,linecolor=black](4,1.5){3pt}{l31}
\cnode[linewidth=0.5pt,fillstyle=solid,fillcolor=white,linecolor=black](6,1.5){3pt}{l32}
\cnode[linewidth=0.5pt,fillstyle=solid,fillcolor=white,linecolor=black](4,2){3pt}{l41}
\cnode[linewidth=0.5pt,fillstyle=solid,fillcolor=white,linecolor=black](6,2){3pt}{l42}

\ncline[linewidth=0.6pt]{v2}{v1}
\ncline[linewidth=0.6pt]{v2}{v3}
\ncline[linewidth=0.6pt,linestyle=dashed,dash=2pt 2pt]{v4}{v3}
\ncline[linewidth=0.6pt]{v3}{l11}
\ncline[linewidth=0.6pt]{l12}{l11}
\ncline[linewidth=0.6pt]{v3}{l21}
\ncline[linewidth=0.6pt]{l22}{l21}
\ncline[linewidth=0.6pt]{v3}{l31}
\ncline[linewidth=0.6pt]{l32}{l31}
\ncline[linewidth=0.6pt]{v3}{l41}
\ncline[linewidth=0.6pt]{l42}{l41}

\rput(1.75,1){\small $e$}
\rput(5,.5){\small $\cdots$}

\rput(3,2.5){\small $H_i$}
\rput(-1,2.5){\small (b)}

}
\endpspicture
\caption{Adding or deleting an edge may arbitrarily increase the certified domination number.
}\label{fig:edge_removing_adding}
\end{center}
\end{figure}

Adding an edge to a graph
may also arbitrarily increase the certified domination number.
Namely, consider the disconnected graph $H_i$ of order $2i+4$ illustrated in Fig.~\ref{fig:edge_removing_adding}(b).
We have $\gamma_\cer(H_i)= i+2$ and $\gamma_\cer(H_i + e)= 2i+4$.
However, adding an edge to a connected graph does not increase
the certified domination number, that is,
$\gamma_\cer(G+e) \le \gamma_\cer(G)$ for any connected graph~$G$.
To argue this property, we use the following lemma.

\begin{lem}\label{lem:three_properties}
Let $D_\c$ be a $\gamma_\cer$-set of a connected graph $G$ of order $n \ge 2$. Then:
\begin{itemize}
\item[$a)$] Every shadowed vertex in $D_\c$ is a weak support or a leaf.
    \item[$b)$] Every non-leaf neighbour of a shadowed weak support
    is either an illuminated vertex or a shadowed weak support.
\end{itemize}
\end{lem}
\begin{proof} (a) Consider a shadowed vertex $v \in D_\c$.
Suppose on the contrary that $v$ is neither a weak support
nor a leaf in $G$. By minimality of $D_\c$, there are no shadowed strong supports in $D_\c$,
 in particular, $v$ is not a strong support, and thus
all neighbours of $v$ are of degree at least two.  Let $X \subseteq D_\c $
be a maximal subset of shadowed vertices such that
(i) $v \in X$, (ii) the induced subgraph $G[X]$ is connected,
and (iii) none of elements of $X$ is an illuminated vertex
or a shadowed weak support.
Next, for a vertex $x\in X$, define the set $B_G(x) = N_G(x) \setminus X$.
Analogously, define the set $B_G(X) = \bigcup_{x \in X} B_G(x)$.

Observe that by minimality of $D_\c$, 
each vertex $x \in X$ is a non-support vertex,
and by the choice of $X$,
and every element in $B_G(X)$ is either an illuminated vertex
or a shadowed weak support.
\\[2mm]
\underline{Case $1$}: {\em $G[X]$ is a $1$-vertex graph}.
Let $L$ be the set of shadowed leaves within the distance $2$ from $v$.
Then the set $D=D_\c \setminus (L \cup \{v\}\big)$
would be a certified dominating set of $G$ and $|D|<|D_\c|$, a contradiction.
\\[2mm]
\underline{Case $2$}: {\em $|X| \ge 2$ and $\gamma_\cer(G[X]) = |X|$}.
By Theorem~5.\ref{thm:corona}, $G[X]$ is the corona of some connected graph.
Observe that by the choice of $X$ and minimality of $D_\c$,
if $x \in X$ is a leaf of $G[X]$, then the set $B_G(x)$ is non-empty.

Consider now a weak support $s$ in $G[X]$.
Let $L_1$ be the set of leaves of $G[X]$ within the distance at most $2$
from $s$ and let $L_2$ be the set of shadowed leaves of $G$
within the distance $2$ from $L_1 \cup \{s\}$.
Then the set $D=D_\c  \setminus (L_1 \cup L_2  \cup \{s\})$
is a certified dominating set of $G$ and $|D|<|D_\c|$, a contradiction.
\\[2mm]
\underline{Case $3$}: {\em $|X| \ge 3$ and $\gamma_\cer(G[X]) \le |X|-2$} (as the case $\gamma_\cer(G[X]) = |X|-1$ is impossible).
Let $D_X$ be a $\gamma_\cer$-set of $G[X]$ and let $\overline{D}_X = X \setminus D_X$.
Let $L_3$ be the set of shadowed leaves within the distance $2$
from $\overline{D}_X$.
Then the set $D=D_\c \setminus (\overline{D}_X \cup L_3)$
is a certified dominating set of $G$ and $|D|<|D_\c|$,
a contradiction.
\\[2mm]
(b) A non-leaf neighbour of a shadowed weak support $s \in  D_\c$
is either illuminated or shadowed. If $s$ is shadowed,
then, since it is not a leaf, it must be a weak support by (a).
\end{proof}

\begin{thm}\label{thm:gammacer_Ge}
If $G$ is a connected graph of order $n \ge 2$, then $\gamma_\cer(G+e) \le \gamma_\cer(G)$.
\end{thm}

\begin{proof}
One can verify the validity of the theorem for graphs of order at most $n \le 4$.
So assume $n \ge 5$ and let $D_\c$ be a $\gamma_\cer$-set of $G$.

Let $e=vw$, $v,w \in V_G$,
be the added edge to $G$. If both $v,w \in D_\c$,
then $D_\c$ is also a certified dominating set of the graph $G+e$.
Similarly, if either both $v, w \notin D_\c$,
or $v \notin D_\c$ and $w \in D_\c$ is illuminated,
or $v \in D_\c$ is illuminated and $w \notin D_\c$,
then $D_\c$ is a~certified dominating set of $G+e$ as well.
Therefore, in all aforementioned cases,
we have $\gamma_\cer(G+e) \le |D| = \gamma_\cer(G)$ as required.

Without loss of generality assume now that $v \notin D_\c$ and $w \in D_\c$ is shadowed (the case $w \notin D_\c$ and $v\in D_\c$ is shadowed can be analysed in a similar way).
By Lemma~6.\ref{lem:three_properties}(a), $w$~is either a weak support or a leaf of $G$.
\\[2mm]
\underline{Case $1$}: {\em $w$ is a weak support of $G$}.
Then the set $D = D_\c \setminus \{\leaf_G(w)\}$ is a certified
dominating set of $G+e$, and thus, $\gamma_\cer(G+e) \le |D| < |D_\c|= \gamma_\cer(G)$.
%
\\[2mm]
\underline{Case $2$}: {\em $w$ is a leaf of $G$}. By the choice of $D_\c$,
it follows that the support vertex $\supp_G(w)$ is weak and shadowed.
Therefore, by Lemma~6.\ref{lem:three_properties}(b),
every non-leaf neighbour of the weak support $\supp_G(w)$ in $G$
is either an illuminated vertex or a shadowed weak support of $G$.
Let $L$ be the set of shadowed leaves within the distance $2$ from $\supp_G(w)$ in $G$.
Then, the set $D = D_\c \setminus (L \cup \{\supp_G(w)\})$
is a certified dominating set in $G+e$,
and hence $\gamma_\cer(G+e) \le |D| \le  |D_\c| -1 < \gamma_\cer(G)$.
\end{proof}

\begin{figure}[!h]
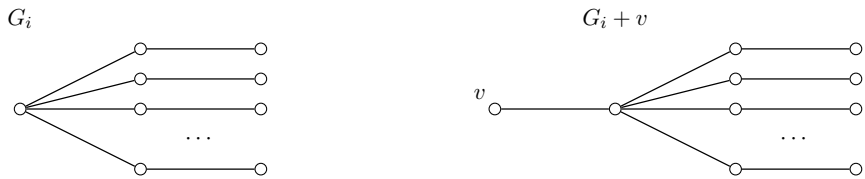
\begin{center}

\pspicture(2,-0.2)(9.8,1.84)\scalebox{0.8}{

\cnode[linewidth=0.5pt,fillstyle=solid,fillcolor=white,linecolor=black](2,1){3pt}{v3}
\cnode[linewidth=0.5pt,fillstyle=solid,fillcolor=white,linecolor=black](4,0){3pt}{l11}
\cnode[linewidth=0.5pt,fillstyle=solid,fillcolor=white,linecolor=black](6,0){3pt}{l12}
\cnode[linewidth=0.5pt,fillstyle=solid,fillcolor=white,linecolor=black](4,1){3pt}{l21}
\cnode[linewidth=0.5pt,fillstyle=solid,fillcolor=white,linecolor=black](6,1){3pt}{l22}
\cnode[linewidth=0.5pt,fillstyle=solid,fillcolor=white,linecolor=black](4,1.5){3pt}{l31}
\cnode[linewidth=0.5pt,fillstyle=solid,fillcolor=white,linecolor=black](6,1.5){3pt}{l32}
\cnode[linewidth=0.5pt,fillstyle=solid,fillcolor=white,linecolor=black](4,2){3pt}{l41}
\cnode[linewidth=0.5pt,fillstyle=solid,fillcolor=white,linecolor=black](6,2){3pt}{l42}

\ncline[linewidth=0.6pt]{v3}{l11}
\ncline[linewidth=0.6pt]{l12}{l11}
\ncline[linewidth=0.6pt]{v3}{l21}
\ncline[linewidth=0.6pt]{l22}{l21}
\ncline[linewidth=0.6pt]{v3}{l31}
\ncline[linewidth=0.6pt]{l32}{l31}
\ncline[linewidth=0.6pt]{v3}{l41}
\ncline[linewidth=0.6pt]{l42}{l41}

\rput(5,.5){\small $\cdots$}

\rput(2,2.5){\small $G_i$}

}
\endpspicture
\pspicture(2,-0.2)(4.8,1.84)\scalebox{0.8}{

\cnode[linewidth=0.5pt,fillstyle=solid,fillcolor=white,linecolor=black](0,1){3pt}{v1}
\cnode[linewidth=0.5pt,fillstyle=solid,fillcolor=white,linecolor=black](2,1){3pt}{v3}
\cnode[linewidth=0.5pt,fillstyle=solid,fillcolor=white,linecolor=black](4,0){3pt}{l11}
\cnode[linewidth=0.5pt,fillstyle=solid,fillcolor=white,linecolor=black](6,0){3pt}{l12}
\cnode[linewidth=0.5pt,fillstyle=solid,fillcolor=white,linecolor=black](4,1){3pt}{l21}
\cnode[linewidth=0.5pt,fillstyle=solid,fillcolor=white,linecolor=black](6,1){3pt}{l22}
\cnode[linewidth=0.5pt,fillstyle=solid,fillcolor=white,linecolor=black](4,1.5){3pt}{l31}
\cnode[linewidth=0.5pt,fillstyle=solid,fillcolor=white,linecolor=black](6,1.5){3pt}{l32}
\cnode[linewidth=0.5pt,fillstyle=solid,fillcolor=white,linecolor=black](4,2){3pt}{l41}
\cnode[linewidth=0.5pt,fillstyle=solid,fillcolor=white,linecolor=black](6,2){3pt}{l42}

\ncline[linewidth=0.6pt]{v3}{v1}
\ncline[linewidth=0.6pt]{v3}{l11}
\ncline[linewidth=0.6pt]{l12}{l11}
\ncline[linewidth=0.6pt]{v3}{l21}
\ncline[linewidth=0.6pt]{l22}{l21}
\ncline[linewidth=0.6pt]{v3}{l31}
\ncline[linewidth=0.6pt]{l32}{l31}
\ncline[linewidth=0.6pt]{v3}{l41}
\ncline[linewidth=0.6pt]{l42}{l41}

\rput(-0.25,1.25){\small $v$}
\rput(5,.5){\small $\cdots$}

\rput(2,2.5){\small $G_i+v$}

}
\endpspicture
\caption{Graph $G_i$ has $2i+1$ vertices, and $\gamma_\cer(G_i)= i$, while $\gamma_\cer(G_i +v)= 2i+2$.}\label{fig:vertex_adding_in}
\end{center}
\end{figure}

As regards adding a vertex, we claim that it may arbitrarily increase the certified domination number
which is not the case as in the model of classic domination.
Indeed, for the graph $G_i$ of order $2i+1$ depicted in Fig.~\ref{fig:vertex_adding_in}, we have $\gamma_\cer(G_i)= i$, while $\gamma_\cer(G_i +v)= 2i+2$.
However, bearing in mind Corollary~4.\ref{wn:gammacer_is_gamma1},
one can expect that the clue of the above construction lies in adding a leaf.
Indeed, this is the case since one can prove that adding a non-leaf vertex
does not effect the certified domination number significantly.
Namely, we have the following theorem.
\begin{thm}\label{thm:gammacer_Gv_deg2}
If we add a non-leaf vertex $v$ to a graph $G$, then $\gamma_\cer(G+v) \le \gamma_\cer(G)+1$.
\end{thm}
\begin{proof}
Let $D_\c$ be a $\gamma_\cer$-set of a graph $G$ and let $v$ be a new added vertex.
\\[2mm]
\underline{Case $1$}: {\em $\deg_{G+v}(v)=2$}. Let $u$ and $w$ be
the two neighbours of $v$ in $G+v$. If either $v,w \notin D_\c$ or both $v,w \in D_\c$,
then the set $D_\c \cup \{v\}$ is a certified dominating set of $G+v$, and
thus $\gamma_\cer(G+v) \le \gamma_\cer(G)+1$.
Otherwise, without loss of generality, we consider two subcases.
\\[2mm]
\underline{Subcase $1$.a}: {\em  $u \notin D_\c$, $w \in D_\c$, and $w$ is illuminated.}
Then the set $D_\c$ remains a certified dominating set of $G+v$,
and in this case, $\gamma_\cer(G+v) \le \gamma_\cer(G)$ holds.
\\[2mm]
\underline{Subcase $1$.b}: {\em  $u \notin D_\c$, $w \in D_\c$, and $w$ is shadowed.}
If  the vertex $w$ constitutes a $1$-vertex component of $G$,
then  the set $(D_\c \setminus \{w\}) \cup \{v\}$
 is  a certified dominating set in $G+v$, thus getting $\gamma_\cer(G+v) \le \gamma_\cer(G)+1$.
Otherwise, by Lemma~6.\ref{lem:three_properties}(a),
$w$ is either a weak support or a leaf of $G$.
Now, similarly as in the proof of Theorem~6.\ref{thm:gammacer_Ge},
we consider two subcases.
\\[2mm]
\underline{Subcase $1$.b.$1$}: {\em  $w$ is a weak support of $G$}.
(We emphasize that this subcase includes the case when $w$ and the $\leaf_G(w)$
constitute a $2$-vertex component of $G$.)
Then the set $D_\c \setminus \{\leaf_G(w)\}$ is a certified
dominating set in $G+v$. In this case, $\gamma_\cer(G+v) \le \gamma_\cer(G)-1$ holds.
\\[2mm]
\underline{Subcase $1$.b.$2$}: {\em $w$ is a leaf of $G$, and
$w$ together with the support vertex} $\supp_G(w)$ {\em does not constitute a $2$-vertex component of $G$}.
By the choice of $D_\c$,
the support vertex $\supp_G(w)$ is weak and shadowed.
By Lemma~6.\ref{lem:three_properties}(b), every non-leaf neighbour of $\supp_G(w)$ in $G$
is either an illuminated vertex or a shadowed weak support of $G$.
Again, let $L$ be the set of shadowed leaves within the distance $2$ from $\supp_G(w)$ in $G$.
Then, the set $D = D_\c \setminus (L \cup \{\supp_G(w)\})$
is a~certified dominating set in $G+v$.
In this case, $\gamma_\cer(G+v) \le \gamma_\cer(G)-1$.
\\[2mm]
\underline{Case $2$}: {\em $\deg_{G+v}(v) \ge 3$}. Then, when adding $v$ to $G$, we first add only two
edges, thus obtaining a temporary graph $G'$, where $\deg_{G'}(v)=2$.
Now, taking into account Case 1, we conclude that $\gamma_\cer(G') \le \gamma_\cer(G)+1$.
Next, when adding all the remaining edges to $G'$, sequentially, to obtain the final graph $G+v$,
we apply Theorem~6.\ref{thm:gammacer_Ge}, sequentially, for each of added edge,
thus getting $\gamma_\cer(G+v) \le \gamma_{\cer}(G') \le \gamma_\cer(G)+1$ as required.
\end{proof}

\section{Nordhaus-Gaddum type results}\label{sec:NG}
Following the precursory paper of Nordhaus and Gaddum~\cite{NG56},
the literature has became abundant in
inequalities of a similar type for many graph invariants,
see a recent survey by~Aouchiche and Hansen~\cite{AH13}.
In particular, the following result is known for the domination number.
\begin{thm}\label{thm:ALL}{\em\cite{B76,CH77,JP72}}
If $\overline{G}$ is the complement of a graph $G$ of order $n$, then:
\begin{itemize} \item[a)] $\gamma(G)\,  \gamma(\overline{G}) \le n$;
\item[b)] $\gamma(G) +\gamma(\overline{G}) \le n+1$ with equality if and only
if $G = K_n$ or $\overline{G} =  K_n$.\end{itemize}\end{thm}

Further sharpening of bounds was done for the case when, for example,
both $G$ and $\overline{G}$ are connected~\cite{LP85}
or for graphs with specified minimum degree~\cite{DHH05}, to mention but a~few. In particular, the following theorem was proved in~\cite{JA95}.

\begin{thm}\label{thm:NOISOLATES}{\em\cite{JA95}}
If a graph $G$ is a graph of order $n$
and neither $G$ nor $\overline{G}$
has an isolated vertex, that is $1 \le \delta(G) \le n-2$,
then $$\gamma(G)+\gamma(\overline{G})\le \left\lfloor\frac{n}{2}\right\rfloor + 2.$$
Moreover, if $n \neq 9$, the bound is attained if and only if\/
$\{\gamma(G), \gamma(\overline{G})\} =
\{\left\lfloor\frac{n}{2}\right\rfloor,2\}$.
\end{thm}

In this section we provide some Nordhaus-Gaddum
type inequalities for the certified domination number.
First, taking into account Corollary~4.\ref{wn:gammacer_is_gamma1},
Theorem~7.\ref{thm:ALL} and Theorem~7.\ref{thm:NOISOLATES},
we obtain the following corollary.

\begin{wn}\label{wn:NG52}
If $G$ is a graph of order $n$ and $\min\{\delta(G),\delta(\overline{G})\} \ge 2$, then
$$\gamma_\cer(G)+\gamma_\cer(\overline{G})\le \left\lfloor\frac{n}{2}\right\rfloor + 2 \ \ \ \textnormal{and} \ \ \  \gamma_\cer(G) \, \gamma_\cer(\overline{G})\leq\ n.$$
\end{wn}

By enumerating all graphs of order at most $4$,
we obtain the following observation.

\begin{obs}\label{lem:NG234}
Let $G$ be a graph of order $n$. Then:
\begin{itemize}
\item[$a)$]  $\gamma_\cer(G)+\gamma_\cer(\overline{G}) = \gamma_\cer(G)\, \gamma_\cer(\overline{G}) = 4$ if $n=2$;
\item[$b)$] $\gamma_\cer(G)+\gamma_\cer(\overline{G}) = 4$ and $\gamma_\cer(G)\, \gamma_\cer(\overline{G}) = 3$ if $n=3$;
\item[$c)$]
$\big(\gamma_\cer(G)+\gamma_\cer(\overline{G}), \gamma_\cer(G)\, \gamma_\cer(\overline{G})\big) \in \{(3,2),(5,4),(6,6),(8,16)\}$ if $n=4$.
\end{itemize}
\end{obs}

Next, we have the following theorem.

\begin{thm}\label{thm:NG50} If $G$ is a graph of order $n \ge 3$ and $\min\{\delta(G),
 \delta(\overline{G})\}=0$, then  $$\gamma_\cer(G)+\gamma_\cer( \overline{G})\leq n+1 \ \ \ \textnormal{and} \ \ \  \gamma_\cer(G) \, \gamma_\cer(\overline{G})\leq n.$$
In addition, if  $\min\{\delta(G),  \delta(\overline{G})\}=0$, then each of the above upper bounds is attainable, and the following statements are equivalent: \begin{itemize}
\item[$a)$] $\gamma_\cer(G)+\gamma_\cer(\overline{G})=n+1$; \item[$b)$]
$\gamma_\cer(G)\, \gamma_\cer(\overline{G})=n$;
\item[$c)$] $G$ or $\overline{G}$ is the complement of $K_n$ or the union of
the corona of some graph and a~positive number of isolated vertices.
    \end{itemize}\end{thm}

\begin{proof} From the assumption  $\min\{\delta(G), \delta(\overline{G} )\}=0$ it follows that $\max\{\Delta(\overline{G}), \Delta(G)\}=n-1$ and, therefore,
$\gamma_\cer(\overline{G})=1$ or $\gamma_\cer({G})=1$. Now, since $\gamma_\cer({G})
\le n$ and $\gamma_\cer(\overline{G})\le n$, we get $\gamma_\cer(G)+\gamma_\cer( \overline{G})\leq n+1$ and $\gamma_\cer(G)\, \gamma_\cer( \overline{G})\leq n$.

Assume now that $\min\{\delta(G), \delta(\overline{G} )\}=0$, say $\delta(G)=0$.
Then $\Delta(\overline{G})=n-1$, and so $\gamma_\cer(\overline{G})=1$. Now,
since $\gamma_\cer(\overline{G})=1$, it follows from each of the equalities $\gamma_\cer(G)+\gamma_\cer(\overline{G})=n+1$ and $\gamma_\cer(G) \, \gamma_\cer(\overline{G})=n$ that $\gamma_\cer({G})=n$. Finally, since
$\delta(G)=0$ and  $\gamma_\cer({G})=n$, we conclude from Theorem~5.\ref{thm:corona} that $G$
is the complement of $\overline{K}_n$ or the union of the corona of some graph and a~positive number of isolated vertices. This proves the implications $a) \Rightarrow
c)$ and $b)\Rightarrow c)$. Opposite implications are straightforward. \end{proof}

Finally, we have the following theorem.

\begin{thm} If $G$ is a graph of order $n\ge 5$,
then \[\gamma_\cer(G)+\gamma_\cer(\overline{G})\le n+2
\quad \mbox{and}\quad \gamma_\cer(G)\, \gamma_\cer(\overline{G})\le 2n.\] In addition, each of the above upper bounds is attainable, and the following statements are equivalent: \begin{itemize}
\item[$a)$] $\gamma_\cer(G)+\gamma_\cer(\overline{G})=n+2$; \item[$b)$]
$\gamma_\cer(G)\, \gamma_\cer(\overline{G})=2n$;
\item[$c)$] $G$ or $\overline{G}$ is the corona of some graph. \end{itemize}\end{thm}

\begin{proof}
If $\min\{\delta(G),\delta(\overline{G})\} \ge 2$, then
$\gamma_\cer(G)+\gamma_\cer(\overline{G})\le \left\lfloor\frac{n}{2}\right\rfloor + 2 \le n+2$ and $\gamma_\cer(G) \, \gamma_\cer(\overline{G})$
$\leq n \le 2n$ by Corollary~7.\ref{wn:NG52}. 
If $\min\{\delta(G),\delta(\overline{G})\} = 0$, then
$\gamma_\cer(G)+\gamma_\cer(\overline{G})\le n + 1 \le n+2$ and $\gamma_\cer(G) \, \gamma_\cer(\overline{G})\leq n \le 2n$ by Theorem~7.\ref{thm:NG50}.

Thus assume $\min\{\delta(G), \delta(\overline{G})\}=1$. Then $\max\{\Delta(G), \Delta(\overline{G})\}=n-2$. This also implies that
$\gamma_\cer(G)>1$ and $\gamma_\cer( \overline{G})>1$. Thus, since $\gamma_\cer(G)\le n$ and $\gamma_\cer( \overline{G})\le n$, it suffices to show that $\gamma_\cer(G)= 2$ or $\gamma_\cer( \overline{G}) = 2$. Without loss of generality assume that $\delta(G)=1$. Let $l$ be a leaf of $G$ and let $s$ be the only element of $N_G(l)$. We consider two cases: $\deg_G(s)=n-2$, and $\deg_G(s)\le n-3$.
\\[2mm]
\underline{Case 1}: {\em  $\deg_G(s)=n-2$}. Let $t$ be the only element of $V_G\setminus N_G[s]$.
Assume first that $d_G(t)\ge 2$. Let $u$ and $w$ be two neighbours of $t$ (and $s$).
Now, because $N_G[\{s,t\}]= N_G[s]\cup N_G[t]= (V_G\setminus \{t\})\cup N_G[t]=V_G$,
$\{u,w\}\subseteq N_G(s)\cap (V_G\setminus\{s,t\})$, $\{u,w\}\subseteq N_G(t)\cap (V_G\setminus\{s,t\})$, and $\gamma_\cer(G)>1$, we conclude that $\{s, t\}$ is a minimum certified dominating set of ${G}$, and $\gamma_\cer(G)=2$. Assume now that
$d_G(t)=1$. In this case, let $u$ and $w$ be vertices such that $N_G(t)=\{u\}$ and
$w\in N_G(s)\setminus \{l,u\}$. Since $N_{\overline{G}}[\{l,t\}]= N_{\overline{G}}[l]\cup N_{\overline{G}}[t]= (V_G\setminus \{s\})\cup (V_G\setminus \{u\})= V_G$,  the set $\{l,t \}$ is dominating in $\overline{G}$. In addition, since $\{u, w\}\subseteq N_{\overline{G}}(l) \cap (V_G\setminus \{l,t\})$ and   $\{s, w\}\subseteq N_{\overline{G}}(t) \cap (V_G\setminus \{l,t\})$, the set $\{l,t \}$ is certified dominating in $\overline{G}$. From this and from the fact that $\gamma_\cer(\overline{G})>1$
it follows that $\gamma_\cer(\overline{G}) = 2$.
\\[2mm]
\underline{Case 2}: {\em  $\deg_G(s)\le n-3$}.
Let $t$ and $u$ be two elements of the
set $V_G\setminus N_G[s]$. In this case, $\{l, s\}$ is a certified dominating set of $\overline{G}$, since $N_{\overline{G}}[\{l, s\}]=V_G$,
$\{t,u\}\subseteq N_{\overline{G}}(l) \cap (V_G\setminus \{l,s\})$, and
$\{t,u\}\subseteq N_{\overline{G}}(s) \cap (V_G\setminus \{l,s\})$. From this it
again follows that $\gamma_\cer(\overline{G}) = 2$.
\\[2mm]
\indent We now prove the equivalence of a), b), and c). Let $G$ be a graph of order $n\ge 5$
such that $\gamma_\cer(G)+\gamma_\cer(\overline{G})=n+2$ ($\gamma_\cer(G)\, \gamma_\cer(\overline{G})=2n$, respectively). From this assumption, from Corollary~7.\ref{wn:NG52} and Theorem~7.\ref{thm:NG50} it follows that $\min\{\delta(G), \delta(\overline{G})\}=1$. Then, as we have already proved, $\gamma_\cer(G)= 2$ or $\gamma_\cer( \overline{G}) = 2$, and therefore  $\gamma_\cer( \overline{G}) = n$ or $\gamma_\cer(G)= n$, respectively. From this and from Theorem~5.\ref{thm:corona} it follows that $\overline{G}$ or $G$ is the corona of some graph.
Thus, we have proved the implications $a)\Rightarrow c)$ and $b)\Rightarrow c)$.
Finally, assume that $G$ is the corona of some graph and $G$ is of order $n\ge 5$. Then $\gamma_\cer(G)= n$ by Theorem~3.\ref{thm:corona}. From the fact that the corona has no
isolated vertex, it follows that $\gamma_\cer( \overline{G})>1$. Now, since $\delta(G)=1$, as in Case 2, we get $\gamma_\cer( \overline{G})=2$. Consequently,  $\gamma_\cer(G)+\gamma_\cer(\overline{G}) = n+2$  and $\gamma_\cer(G) \, \gamma_\cer(\overline{G}) = 2n$. This proves the implications  $c)\Rightarrow a)$ and $c)\Rightarrow b)$. \end{proof}

\section{The complexity status}\label{sec:NP-hardness}
When introducing a new model of domination,
no discussion would be complete without mentioning
its complexity status. Define the (decision) certified domination problem
as follows.

\bigskip\noindent
\begin{tabularx}{\textwidth}{lX}
  \toprule
  \multicolumn{2}{l}{The certified domination problem}\\
  \midrule
  \bfseries{Instance}: & A graph $G$ and an integer $k \ge 1$.\\
  \bfseries{Question}: & Does there exist a certified dominating set of $G$
  of cardinality at most $k$?\\
  \bottomrule
\end{tabularx}

\bigskip
As one can expect, the certified domination problem is NP-complete.
Specifically, the certified domination problem remains
NP-complete even if we restrict ourselves to bipartite planar subcubic graphs. The~proof immediately follows from Corollary~4.\ref{wn:gammacer_is_gamma1} and
the fact that the (original) domination problem is NP-complete
in bipartite planar subcubic graphs with no leaves~\cite{KYK80,Su10}.

\begin{thm}\label{thm:NP}
The certified domination problem is NP-complete
in bipartite planar subcubic graphs with no leaves.
\end{thm}

Following several algorithmic results for
the domination problem and its variations~\cite{Ch13,HHS98,K13},
one can focus on designing effective (exponential time) algorithms,
faster (polynomial time) algorithms for some restricted graph classes,
or approximation algorithms 
for the variant of certified domination.

\section{Certified domination and $D\!D_2$-graphs}\label{sec:DD2}
Here, we briefly continue our discussion from the introductory section
on the relation between the concepts of certified domination and $D\!D_2$-graphs.
In~\cite{HR13}, the following theorem was established.

\begin{thm}{\em \cite{HR13}}\label{thm:HR13}
A graph of a minimum degree at least two is a $D\!D_2$-graph.
\end{thm}

Since any minimal dominating set has no shadowed vertices,
one can observe that if a~graph $G$ has a~minimal dominating set $D$
being also certified 
then its complement $V_G \setminus D$ is a $2$-dominating set of~$G$.
Therefore, taking into account Corollaries~4.\ref{wn:gammacer_is_gamma1} and~4.\ref{wn:gammacer_is_gamma2},
we conclude with the following theorem
which generalizes Theorem~9.\ref{thm:HR13}.

\begin{thm}\label{thm:gammacer_gamma_DD2}
\begin{itemize}
\item[$a)$] A graph of a minimum degree at least one and without weak supports is a $D\!D_2$-graph.
\item[$b)$] A graph of order at least two having a unique dominating set is a $D\!D_2$-graph.
\end{itemize}
Moreover, any such graph in $a)$ or $b)$ has a $(D,D_2)$-pair such that $|D|=\gamma(G)$.\eop
\end{thm}

\section{Concluding remarks}
Since over the years researchers have published thousands of papers
on the topic of domination in graphs, our paper cannot claim the right
to cover the new model even partially,
it should only be thought of as a very beginning, a small contribution to.
In this section, we present two exemplary open problems
that we find interesting and which research on we feel worth of being continued.

\paragraph{Stable and critical graphs.}
It is natural to characterise the class of {\em critical}  graphs where the certified domination
number increases on the removal of any edge/vertex as well as the class of {\em stable} graphs where the certified domination number remains unchanged on the removal of any edge/vertex.
We point out that by Corollary~4.\ref{wn:gammacer_is_gamma1},
the class of critical (resp.\ stable) (with respect to the certified domination number) graphs
with minimum degree $\delta \ge 3$ is the same as the class of critical (resp.\ stable) graphs
with respect to domination number, see for example~\cite{BCD88,FHG95,MK12}.
Therefore, we are left with characterising critical (resp.\ stable) graphs
with minimum degree $\delta \le 2$. This is an open problem.

\paragraph{Trees with $\gamma_\cer=\gamma$.}
The problem of constructive characterisations of trees with equal
domination parameters has received attention
in the literature, see for example~\cite{C10, H08, LHXL10},
to mention but a few recent. Following this concept,
we leave as an open problem to provide a constructive characterisation of
$(\gamma,\gamma_\cer)$-trees,
that is, the class of trees with $\gamma_\cer=\gamma$.

\end{document}